\documentclass{article}

\usepackage{arxiv}

\usepackage[utf8]{inputenc} 
\usepackage[T1]{fontenc}    
\usepackage{hyperref}       
\usepackage{url}            
\usepackage{booktabs}       
\usepackage{amsfonts}       
\usepackage{nicefrac}       
\usepackage{microtype}      
\usepackage{lipsum}

\usepackage{url}
\usepackage{array,tabularx}
\usepackage{multicol, nonfloat}  
\usepackage{epsfig} 
\usepackage{amssymb}  
\usepackage{amsthm}
\newtheorem{theorem}{Theorem}[section]

\newtheorem{lemma}[theorem]{Lemma}
\newtheorem{remark}{Remark}
\newtheorem{definition}{Definition}[section]
\newtheorem{proposition}[theorem]{Proposition}
\usepackage{blindtext}
\usepackage[utf8]{inputenc}	
\usepackage[english]{babel}
\usepackage{graphicx}
\usepackage{amsmath}
\usepackage{amsfonts}
\usepackage{epstopdf}
\usepackage{cite}
\usepackage{microtype}

\newtheorem{assumption}{Assumption}

\usepackage{color}
\newcommand{\norm}[1]{\| #1 \|}

\newcommand{\quotes}[1]{``#1''}
\newcommand{\etal}{\textit{et al.}}

\usepackage{pgf}

\makeatletter
\newdimen\@widthOfTo%
\newdimen\@widthOfLand%
\newdimen\@widthOfImplies%
\pgfmathsetmacro{\@scaleFactorImplies}{\@widthOfTo/\@widthOfImplies}%
\pgfmathsetmacro{\@scaleFactorTo}{\@widthOfLand/\@widthOfTo}%
\newcommand*{\ScaledImplies}{\mathrel{\raisebox{0.3ex}{\scalebox{\@scaleFactorImplies}{\ensuremath{\Longrightarrow}}}}}%
\newcommand*{\ScaledTo}{\mathbin{\raisebox{0.3ex}{\scalebox{\@scaleFactorTo}{\ensuremath{\to}}}}}%
\makeatother

\title{On the Exponential Stability of Projected Primal-Dual Dynamics on a Riemannian Manifold}

\author{
  P. A. Bansode\thanks{P. A. Bansode is with department of Instrumentation Engineering, Ramrao Adik Institute of Technology, Mumbai, 400706 India.~{\tt\small prashant.bansode@rait.ac.in}} \\
   \And
 V. Chinde\thanks{V. Chinde is with National Renewable Energy Laboratory, Lakewood, CO 80401, USA.}  \\
 \And
 S. R. Wagh\thanks{S. R. Wagh is with department of Electrical Engineering, Veermata Jijabai Technological Institute, Mumbai, 400019 India.} \\
 \And
 R. Pasumarthy\thanks{R. Pasumarthy is with department of Electrical Engineering, Indian Institute of Technology Madras, Madras, 600036 India.} \\
 \And
 N. M. Singh\thanks{N. M. Singh is with department of Electrical Engineering, Veermata Jijabai Technological Institute, Mumbai, 400019 India.} 
}

\begin{document}
\maketitle


\begin{abstract}
	Equivalence of convex optimization and variational inequality is well established in the literature such that the latter is formally recognized as a fixed point problem of the former. Such equivalence is also known to exist between a saddle-point problem and the variational inequality. The variational inequality is a static problem which can be further studied within the dynamical settings using a framework called the projected dynamical system whose stationary points coincide with the static solutions of the associated variational inequality. Variational inequalities have rich properties concerning the monotonicity of its vector-valued map and the uniqueness of its solution, which can be extended to the convex optimization and saddle-point problems. Moreover, these properties also extend to the representative projected dynamical system. The objective of this paper is to harness rich monotonicity properties of the representative projected dynamical system to develop the solution concepts of the convex optimization problem and the associated saddle-point problem. To this end, this paper studies a linear inequality constrained convex optimization problem and models its equivalent saddle-point problem as a variational inequality. Further, the variational inequality is studied as a projected dynamical system\cite{friesz1994day} which is shown to converge to the saddle-point solution. By considering the monotonicity of the gradient of Lagrangian function as a key factor, this paper establishes exponential convergence and stability results concerning the saddle-points. Our results show that the gradient of the Lagrangian function is just monotone on the Euclidean space, leading to only Lyapunov stability of stationary points of the projected dynamical system.  To remedy the situation, the underlying projected dynamical system is formulated on a Riemannian manifold whose Riemannian metric is chosen such that the gradient of the Lagrangian function becomes strongly monotone. Using a suitable Lyapunov function, the stationary points of the projected dynamical system are proved to be globally exponentially stable and convergent to the unique saddle-point. 
\end{abstract}


\section{Introduction}
The convex optimization methods have remained as the subject of substantial research for many decades. The primal-dual gradient-based method is one of such methods which dates back to late $1950s$\cite{arrow1958studies}. Lately, these methods (also referred to as primal-dual dynamics as its dynamical system equivalent) have found many applications in the networked systems (viz, the power networks \cite{zhao2014design,mallada2017optimal,yi2015distributed} and the wireless networks  \cite{feijer2010stability,chen2012convergence,ferragut2014network}),  and building automation systems \cite{kosaraju2018stability}. The primal-dual dynamics (PD dynamics) seek a solution to the saddle-point problem representing the original constrained convex optimization problem by taking gradient descent along the primal variable and gradient ascent along the dual variable. From the perspectives of systems and control theory, the PD dynamics have much to offer in terms of stability and convergence with respect to the saddle point solution.

During recent years the notions of stability of PD dynamics have evolved. The asymptotic stability of PD dynamics has been established as one of the most fundamental notions. Feijer \etal \cite{feijer2010stability} explores the PD dynamics with applications to network optimization problems and prove its asymptotic stability. The dual dynamics pertaining to the inequality constraints have been shown to include switching projections that restrict the dual variables to the set of nonnegative real numbers. Due to the switching projections, the PD dynamics becomes discontinuous, which is further modeled as a hybrid dynamical system. A Krasovskii-type Lyapunov function along with LaSalle invariance principle of hybrid systems \cite{lygeros2003dynamical} have been utilized to prove the asymptotic stability of the PD dynamics. In \cite{cherukuri2016asymptotic} it is proved that the PD dynamics is a special case of the projected dynamical systems. It uses the invariance principle of Carath\'eodory solutions to show that the saddle-point solution of PD dynamics is unique and globally asymptotically stable. Although widely established, the notion of asymptotic stability does not offer explicit convergence bounds of the PD dynamics which is an essential factor in case of the on-line optimization techniques. One must ensure that the trajectories converge to the saddle point solution in finite time. To explicitly obtain stronger convergence rates, research interests have shifted towards the notions of global exponential convergence and stability.

The pathway leading to the exponential stability of the PD dynamics is not as straightforward as it is for its asymptotic stability. The existence of right-hand side discontinuities and non-strongly monotone gradient of the associated Lagrangian function seem to prevent the saddle-point solution from being exponentially stable. The globally exponential stability has been the most desirable yet often formidable aspect of PD dynamics, which guarantees a minimum rate of convergence to the saddle point. 
While exhaustive literature on asymptotic stability of the PD dynamic can be encountered, its exponential stability has not been explored except for the recent studies      \cite{cortes2018distributed,nguyen2018contraction,qu2019exponential,dhingra2018proximal}. The optimization problem considered in \cite{cortes2018distributed} proves the exponential stability of the PD dynamic for an equality constrained optimization problem. Robustness and contraction analysis of the primal-dual dynamics establishing exponential convergence to the saddle-point solution is presented in \cite{nguyen2018contraction}. In \cite{qu2019exponential}, the PD dynamics is proved to be globally exponentially stable for linear equality and inequality constrained convex optimization problem. Under assumptions on strong convexity and smoothness of the objective function and full row rank conditions of the constraint matrices, the PD dynamics is shown to have global exponential convergence to the saddle point solution. It mainly proposes the augmented Lagrangian function that results in a PD dynamics which does not have right-hand side discontinuities. By employing a quadratic Lyapunov function that has non-zero off-diagonal block matrices, it shows that the PD dynamics is globally exponentially stable. 
In \cite{dhingra2018proximal} a composite optimization problem is considered in which the objective function is represented as a sum of differentiable non-convex component and convex non-differentiable regularization component. A continuously differentiable proximal augmented Lagrangian is obtained by using a Moreau envelope of the regularization component.  This results in a continuous-time PD dynamics which under the assumption of strong convexity of the objective function, is shown to be exponentially stable by employing a framework of the integral quadratic constraints (IQCs) \cite{587335}. By using a well-known result pertaining to linear systems with nonlinearities in feedback connection that satisfy IQCs \cite{hu2016exponential}, it proves the global exponential stability of the PD dynamics.

\subsection{Motivation and contribution}
For a sufficiently small step size, a Euler discretized globally exponentially stable PD dynamics leads to geometric convergence to the saddle-point solution\cite{stuart1994numerical}. This property has been widely appreciated in recent articles such as \cite{dhingra2018proximal,qu2019exponential}. The existing methods have considered the augmented Lagrangian techniques at a pivotal position for proving the globally exponential stability of the PD dynamics. 
This paper does not rely on augmented Lagrangian techniques to arrive at exponentially stable saddle-point solution. It presents a complementary approach that uses a combined framework of variational inequalities, projected dynamical systems \cite{nagurney2012projected}, and the theory of Riemannian manifolds\cite{udriste1994convex,da2002contributions} to derive conditions that lead to the global exponential stability of the saddle-point solution.

This paper exploits an equivalence between a constrained optimization problem and a variational inequality problem as discussed in  \cite{kinderlehrer1980introduction,facchinei2007finite,nagurney2012projected}. They are shown to be equivalent when the vector-valued map of the variational inequality is the gradient of the objective function of the underlying optimization problem. Besides that, when the objective function is convex the Karush-Kuhn-Tucker (KKT) conditions of both problems reveal that the Lagrangian function associated with the variational inequality and the Lagrangian of the optimization problem have the exactly same saddle-point\cite{facchinei2007finite}. This further hints at formulating the saddle-point problem (of the corresponding optimization problem) as a variational inequality when both primal, as well as dual variables, are of interest. 

Variational inequalities are equivalent to fixed point problems \cite{nagurney2012projected} in the sense that they yield only the static solutions. Thus the variational inequality formulation of the saddle-point problem would only result in the static description of the saddle-point. To understand the dynamic behavior of such variational inequality, this paper brings in the framework of projected dynamical systems\cite{friesz1994day,nagurney2012projected}. The projected dynamical system combines essential features of both variational inequalities and dynamical systems such that its solution coincides with the static equilibrium of the variational inequality problem. These dynamical systems have interesting features which they derive from the underlying variational inequality problem. In \cite{gao2003exponential} a globally projected dynamical system \cite{friesz1994day} is proved to be exponential stable when the vector-valued mapping concerning the variational inequality problem is strongly monotone. This motivates to represent the saddle-point problem as a variational inequality and use the framework of the projected dynamical system for proposing a new dynamical system that is equivalent to the PD dynamics. Aiming at exponential stability of the saddle-point solution, this paper indirectly poses the saddle-point problem as a projected dynamical system (regarded hereafter as the projected primal-dual dynamics). While deriving the stability results of the proposed dynamics, our analysis reveals that the gradient of the Lagrangian function is not strongly monotone on the Euclidean space, which further deprives the proposed dynamics of being exponentially stable. Towards this end, our paper seeks a differential geometry which favors the desired properties such as strong monotonicity of the gradient of the Lagrangian and exponential stability of the proposed dynamics.

Convexity and monotonicity properties have strong connections\cite{karamardian1990seven}. Since Riemannian geometry is considerably the most natural framework for convexity\cite{udriste1994convex,helmke2012optimization}, it can also be explored for the monotonicity properties of the underlying gradient maps. There already exists a wide interest in optimization \cite{luenberger1972gradient,da2002contributions,absil2009optimization} and projected dynamical systems \cite{hauswirth2016projected,hauswirth2018projected} over manifolds, which further motivates us to appreciate the Riemannian geometry for the proposed dynamics. For a linear inequality constrained convex optimization problem, our work establishes that the key to achieving an exponentially stable projected primal-dual dynamics is to choose a Riemannian metric that leads to strong monotonicity of the underlying gradient. With Lipschitz continuity of the gradient of the Lagrangian function, it is proved that the equilibrium solution of the projected primal-dual dynamics is globally exponentially stable.

The reported work envelopes following key contributions:
\begin{itemize}
	\item The equivalence between a saddle point problem and a variational inequality problem is established and using the framework of the projected dynamical system, a projected PD dynamic is proposed.
	\item The restriction of the strongly monotone gradient of the Lagrangian function is overcome by proposing the projected PD dynamics constrained to a Riemannian manifold under a suitable Riemannian metric.
	\item The projected PD dynamics defined over a Riemannian manifold is proved to be globally exponentially stable for a linear inequality constrained convex optimization problem. 
	\item The effectiveness of the proposed method is studied with the application of $L_2$ regularized least squares problem. The Euler discretized version of the proposed dynamics is shown to converge geometrically to the saddle-point solution.
\end{itemize}

\subsection{Notations and Preliminaries}\label{prel}
The set $\mathbb{R}$ (respectively $\mathbb{R}_{\geq 0}$ or $\mathbb{R}_{> 0}$) is the set of real (respectively non-negative or positive) numbers. If $f: \mathbb{R}^n \rightarrow \mathbb{R}$ is continuously differentiable in $x \in \mathbb{R}^n$, then $\nabla_x f:\mathbb{R}^n \rightarrow \mathbb{R}^n$ is the gradient of $f$ with respect to $x$. $\norm{.}$ denotes the Euclidean norm. For scalars $x,y$, $[x]^+_y:=x$ if $y>0$ or $x>0$, and $[x]^+_y:=0$ otherwise. For a set $X \subseteq \mathbb{R}^n$, the notation $\mathrm{relint}X$ defines the relative interior of $X$. The notation $\mathrm{P_1} \implies \mathrm{P_2}$ implies that the problem $\mathrm{P_2}$ can be derived by specializing the problem $\mathrm{P_1}$.

The following subsections provide preliminaries relevant to the main results of this paper.
\subsubsection{Convex Optimization, Variational Inequality, and Projected Dynamical System}
\begin{definition}(The Variational Inequality Problem, \cite{nagurney2012projected})\label{vi_Def}\\
	For a closed convex set $X\in \mathbb{R}^n$ and vector function $F:X\rightarrow \mathbb{R}^n$, the finite dimensional variational inequality problem, $\mathrm{VI(F,X)}$, is to determine a vector $x^*\in X$ such that 
	\begin{equation}
	(x-x^*)^TF(x^*)\geq 0,~\forall x\in X. \label{vip}
	\end{equation}
\end{definition}
The $\mathrm{VI(F,X)}$ is equivalent to solving a system of equations as given below:

\begin{proposition}(A System of Nonlinear Equations,\cite{nagurney2012projected})\label{snle}\\
	Let $F: X \rightarrow \mathbb{R}^n$ be a vector function. Then $x^* \in \mathbb{R}^n$ solves the variational inequality $\mathrm{VI(F,X)}$ if and only if $x^*$ solves the system of equations
	\begin{align}
	F(x^*)=\mathbf{0}. \label{nle}
	\end{align}	
\end{proposition}
where $\mathbf{0}$ is zero vector of appropriate dimensions.

When the function $F$ is realized as a gradient of a real-valued function $f$, the following relationship between a variational inequality and an optimization problem is established.
\begin{proposition}(An Optimization Problem,\cite{nagurney2012projected})\label{optpro}\\
	Let $X\subset \mathbb{R}^n$ be closed and convex and $f:X \rightarrow \mathbb{R}$ be a continuously differentiable function. If $x^*\in X$ solves the optimization problem:
	\begin{align}
	\min_{x\in X} f(x), \label{opt}
	\end{align}  	
	then $x^*$ solves the variational inequality problem $\mathrm{VI(\nabla f,X)}$. On the other hand, if $f(x)$ is a convex function and $x^*$ solves the $\mathrm{VI(\nabla f,X)}$, then $x^*$ is a solution to the optimization problem \eqref{opt}.
\end{proposition}
The optimization problem \eqref{opt} is denoted hereafter as $\mathrm{OPT(f,X)}$.

If convexity of $f$ holds, then the following relationship between $\mathrm{OPT(f,X)}$ and $\mathrm{VI(\nabla f,X)}$ can be derived.
\begin{align}
\mathrm{VI(\nabla f,X)} \implies \mathrm{OPT(f,X)}. \label{r1}
\end{align}
A variational inequality problem $\mathrm{VI(F,X)}$ is also equivalent to a fixed point problem as given below:
\begin{proposition}(A Fixed Point Problem,\cite{nagurney2012projected})\label{fpppro}\\
	$x^*$ is a solution to $\mathrm{VI(F,X)}$ if and only if for any $\alpha >0$, $x^*$ is a fixed point of the projection map:
	\begin{align}
	x^*=P_X(x^*-\alpha F(x^*)) \label{fppprob}
	\end{align}
	where \begin{align}
	P_{X}=\arg \min_{v\in X}\norm{x-v}.\label{projection}
	\end{align}
\end{proposition}
The fixed point problem, denoted by $\mathrm{FPP(F,X)}$ shares connection with the variational inequality problem $\mathrm{VI(F,X)}$ as described below.
\begin{align}
\mathrm{FPP(F,X)} \implies \mathrm{VI(F,X)}.\label{r2}
\end{align}
\begin{theorem}(Existence of Solution of Variational Inequality,\cite{nagurney2012projected}) \label{thm2.1nagu}\\
	If $X$ is compact and convex and $F(x)$ is continuous on $X$, then the variational inequality problem $\mathrm{VI(F,X)}$ admits at least one solution $x^*$.
\end{theorem}

The monotonicity properties of $F$ required in this paper are stated below:
\begin{definition}{(Monotone Map,\cite{karamardian1990seven})}\label{def1.1}\\
	A mapping $F$ is monotone on $X \subseteq \mathbb{R}^n$, if for every pair of distinct points $x,y\in X$, we have 
	\begin{equation}
	(y-x)^T(F(y)-F(x)) \geq 0.
	\end{equation}  
\end{definition}

\begin{definition}{(Strongly Monotone Map,\cite{karamardian1990seven})}\label{def1.3}\\
	A mapping $F$ is strongly monotone on $X \subseteq \mathbb{R}^n$, if there exists $\mu>0$ such that, for every pair of distinct points $x,y\in X$, we have 
	\begin{equation}
	(y-x)^T(F(y)-F(x)) \geq \mu \norm{x-y}^2.
	\end{equation}  
\end{definition}

The relation between monotonicity of $F$ and positive definiteness of its Jacobian matrix
\begin{align}
\nabla F(x)=\Bigg (\frac{\partial F_i(x)}{\partial x_j}\Bigg)_{i,j=1,2,\ldots,n},
\end{align}
as given below.

\begin{proposition}((Strongly) Positive Definite Jacobian of $F(x)$ implies (Strongly) Monotone $F(x)$,\cite{nagurney2012projected})\label{monoto}\\
	Suppose that $F$ is continuously differentiable on $X$.
	\begin{enumerate}
		\item If the Jacobian matrix $\nabla F(x)$ is positive semidefinite, i.e., 
		\begin{align}
		y^T \nabla F(x) y \geq 0, \forall y \in \mathbb{R}^n, x \in X,
		\end{align}
		then $F$ is monotone on $X$.
		\item If the Jacobian matrix $\nabla F(x)$ is positive definite, i.e., 
		\begin{align}
		y^T \nabla F(x) y > 0, \forall y \in \mathbb{R}^n, x \in X,
		\end{align}
		then $F$ is strictly monotone on $X$.
		\item If $\nabla F(x)$ is strongly positive definite, i.e., 
		\begin{align}
		y^T\nabla F(x) y \geq \mu \norm{z}^2,\forall y \in \mathbb{R}^n, x \in X,
		\end{align}
		then $F(x)$ is strongly monotone on $X$.
	\end{enumerate}	
\end{proposition}

\begin{proposition}(Strongly Monotone $\nabla f$ implies Strongly Convex $f$,{\cite{karamardian1990seven}})\label{prop1}\\
	Let $f$ be a differentiable function on $D \subseteq \mathbb{R}^n$ that contains $X$. Then, $f$ strongly convex if and only if $\nabla f$ is strongly monotone on $X$. 
\end{proposition}

\begin{definition}(Strongly Convex Function\cite{karamardian1990seven})\label{strong_conv_def}\\
	A differentiable function $f$ is strongly convex on a domain $D\subseteq \mathbb{R}^n$ of $X$, if there exists $\mu>0$ such that, for all $x,y\in D$, 
	\begin{align}
	f(y)-f(x)\geq (y-x)^T\nabla f(x)+\frac{\mu}{2}\norm{y-x}^2.\label{scvx}
	\end{align}
	$f$ is strongly concave, if $-f$ is strongly convex.
\end{definition}

\begin{theorem}(Uniqueness of the Solution to Variational Inequality, \cite{nagurney2012projected})\label{uniqueo}\\
	Suppose that $F(x)$ is strongly monotone on $X$. Then there exists precisely one solution $x^*$ to $\mathrm{VI(F,X)}$.
\end{theorem}

Consider the following globally projected dynamical system proposed in \cite{friesz1994day}, denoted here as $\mathrm{PDS(F,X)}$:
\begin{align}
\dot{x}=\beta\{P_{X}[x-\alpha F(x)]-x\}\label{frz}
\end{align}
where $k,\alpha$ are positive constants and $P_{X}:\mathbb{R}^{n}\rightarrow X$ is a projection operator as defined in \eqref{projection}.

\begin{remark}(Equivalence between $\mathrm{PDS(F,X)}$ and $\mathrm{VI(F,X)}$, \cite{gao2003exponential})\label{remark3}\\
	$x^*$ is an equilibrium point of the $\mathrm{PDS(F,X)}$, \eqref{frz} if and only if $x^*$ is a solution of the variational inequality problem $\mathrm{VI(F,X)}$ defined in \eqref{vip}.
\end{remark}
From Remark \ref{remark3}, 
\begin{align}
\dot{x}=0\implies x^* = P_X(x^*-\alpha F(x^*)).
\end{align}
By using Proposition \ref{fpppro}, the result is immediate. The $\mathrm{PDS(F,X)}$ shares connection with the $\mathrm{VI(F,X)}$ as described below:
\begin{align}
\mathrm{PDS(F,X)} \implies \mathrm{FPP(F,X)} \implies \mathrm{VI(F,X)}.\label{r3}
\end{align}
\begin{lemma}\label{lemma1.3}{\cite{xia2000stability}}
	Assume that $F$ is locally Lipschitz continuous in a domain $D$ that contains $X$. Then the solution $x(t)$ of \eqref{frz} will approach exponentially the feasible set $X$ when the initial point $x^0\notin X$. Moreover, if $x^0\in X$, then $x(t) \in X$.
\end{lemma}
\begin{remark}\cite{gao2003exponential}
	Let $X$ be a nonempty, closed convex set, then the following holds:\begin{align}
	[x-P_X(x)]^T[P_X(x)-y]\geq 0,\forall y \in X, x \in \mathbb{R}^{n}. \label{bapro}
	\end{align}
\end{remark}

Given a set $X \subset \mathbb{R}^n$ and $x \in X$, a vector $v \in \mathbb{R}^n$ is called an inward tangent vector of $X$ at $x$ if there exist a smooth curve $\gamma:[0,\xi]\rightarrow X$ such that $\xi \geq 0$, $\gamma(0)=x$, and $\gamma'_+(0)=v$.

The set of all inward tangent vectors is the tangent cone of $X$ at $x$ and denoted by $T_xX$.

Denote the boundary and interior of $X$, respectively, by $\partial X$ and $X^0$. If $x \in \partial X$, the set of inward normals to $X$ at $x$ is defined as the dual cone of $T_xX$, as follows:

\begin{align}
N_xX = \{\eta:\norm{\eta}=1|\langle \eta^T,x-x' \rangle \leq 0, \forall x' \in T_xX\}.
\end{align}
From Definition \ref{vi_Def}, the necessary and sufficient condition for $x^*$ to be a solution to $\mathrm{VI(F,X)}$\cite{nagurney2012projected}, is that
\begin{align*}
-F(x^*) \in N_{x^*}X.
\end{align*}

\subsubsection{Riemannian manifold}
Let $\mathcal{M}$ be a differential manifold endowed with Riemannian metric $r$. Let $\mathcal{S}(\mathcal{M})$ denote the space of vector fields over $\mathcal{M}$. The tangent space of $\mathcal{M}$ at some point $x\in \mathcal{M}$ be $T_x\mathcal{M}$.

An inner product $r$ is defined as $\langle u,v \rangle_r:=r(u,v)$. In matrix form, $\langle u,v \rangle_r=u^TRv$ where $R$ is symmetric positive definite. The $2-$norm induced by $r$ is written as $\norm{.}_r$ such that $\norm{v}=\sqrt{\langle v,v \rangle_r}$.
$r_x:T_x\mathcal{M} \times T_x\mathcal{M}\rightarrow \mathbb{R}$ is a smoothly chosen inner product on the tangent space $T_x\mathcal{M}$ of $\mathcal{M}$. For each $x \in \mathcal{M}$, $r = r_x$, satisfies the following:
\begin{enumerate}
	\item $r(x_1,x_2)=r(x_2,x_1),\forall x_1,x_2 \in T_x\mathcal{M}$
	\item $r(x,x)>0,\forall x\in T_x\mathcal{M}$
	\item $r(x,x) = 0$ if and only if $x=0$.
\end{enumerate}

\begin{definition}(Riemannian Metric,\cite{rockafellar2009variational})\label{def1.7}\\
	Given a set $X\subseteq \mathbb{R}^n$, a Riemannian metric is a map $r:X\rightarrow L^n_2$ that assigns to every point $x \in X$ an inner product $\langle.,.\rangle_{r(x)}$. A metric is Lipschitz continuous if it is continuous as a map from $X$ to $L^n_2$. 
\end{definition}

If there exists a smooth function $f:\mathcal{M}\rightarrow \mathbb{R}$, and $x\in \mathcal{M}$, the differential $D_xf:T_x\mathcal{M}\rightarrow \mathbb{R}$ is defined as
\begin{align*}
D_xf(v)=(f \circ \gamma)'(v)
\end{align*}
where $\gamma(-\xi,\xi)$ is a smooth curve with $\gamma(0)=x$, and $\gamma'(0)=v$, $v\in T_x\mathcal{M}$.
The gradient of $f$ at $x\in \mathcal{M}$ is defined as the unique tangent vector $\mathrm{grad}f$ such that
\begin{align*}
\langle \mathrm{grad}f,v\rangle_r=D_xf(v)\forall v \in T_x\mathcal{M}.
\end{align*}
In matrix notation, the gradient $\mathrm{grad}_rf=R^{-1}\nabla{f}^T$.
A vector field $F$ is a map that assigns a vector $F(x)\in T_x\mathcal{M}$ to every point $x\in \mathcal{M}$.

Let $\nabla$ be a linear (Levi-Civita) connection and $Y$ be a $C^\infty$ vector field on $\mathcal{M}$, respectively. Then the connection $\nabla$ induces a covariant derivative with respect to $Y$, denoted by $\nabla_Y$.

The differential of $F\in \mathcal{S}(\mathcal{M})$ is a linear operator $\mathbb{H}_F:\mathcal{S}(\mathcal{M})\rightarrow\mathcal{S}(\mathcal{M})$, given by $\mathbb{H}_F(Y):=\nabla_Y(F)$. The linear map $\mathbb{H}_F(x):T_x\mathcal{M}\rightarrow\mathcal{M}$ assigned to each point $x\in \mathcal{M}$ is defined by
$\mathbb{H}_F(x)v=\nabla_vF, \forall v \in T_x\mathcal{M}$.

If $F = \mathrm{grad}f$, where $f:\mathcal{M}\rightarrow \mathbb{R}$, then $\mathbb{H}_F(x)$ denotes the Hessian of $f$ at $x$.

\begin{proposition}(Generalization of Strongly Monotone Vector-valued Function on a Riemannian Manifold ,\cite{da2002contributions})\label{prop1.5}\\
	Let $\mathcal{M}$ be a Riemannian manifold and let $F$ be a vector field on $\mathcal{M}$. Then
	$F$ is strongly monotone if and only if there exists $\nu>0$ such that $\langle \mathbb{H}_{F(x)}v,v\rangle_r \geq \nu\norm{v}^2_r$ for any $x\in \mathcal{M}$ and $v \in T_x\mathcal{M}$.
\end{proposition}
\begin{proposition}(Generalization of Strongly Convex Real-valued Function on a Riemannian Manifold,\cite{da2002contributions})\label{prop1.6}\\
	Let $\mathcal{M}$ be a Riemannian manifold, then the function $f:\mathcal{M}\rightarrow \mathbb{R}$ is strongly convex if and only if the gradient vector field $F=\mathrm{grad}f$ on $\mathcal{M}$ is strongly monotone.
\end{proposition}

\section{Saddle-point problem as a variational inequality and projected primal-dual dynamics} \label{pf}Consider the following constrained optimization problem
\begin{equation}
\setlength\arraycolsep{1.5pt}
\begin{array}{cc}
\mathrm{minimize} & f(x)\\
\mathrm{subject~to} & x \in X
\end{array} \label{cvx}
\end{equation}
where 
\begin{align}
X &=\{x\in \mathbb{R}^n|g_i(x)\leq 0,\forall^m_{i=1}\},\label{constr_set} 
\end{align} is the domain of the problem \eqref{cvx}.
The functions $f:\mathbb{R}^n\rightarrow\mathbb{R}$, $g:\mathbb{R}^n \rightarrow \mathbb{R}^m$ are assumed to be continuously differentiable $(\mathcal{C}^2)$ with respect to $x$, with the following assumptions:
\begin{assumption}\label{ass1}
	$\nabla f: \mathbb{R}^n\rightarrow \mathbb{R}^n$ is strongly monotone on $X$, with $\mu>0$  such that the following holds:
	\begin{equation}
	(x_1-x_2)^T(\nabla f(x_1)-\nabla f(x_2)) \geq \mu \norm{x_1-x_2}^2.
	\end{equation}  
\end{assumption} 
As a consequence of Assumptions \ref{ass1}, it is derived that the objective function $f$ is strongly convex in $x$ with the modulus of convexity given by $\frac{\mu}{2}$.

\begin{assumption}\label{ass2}
	Constraints $g_i(x)$ are convex in $x,\forall^m_{i=1}$.
\end{assumption} 
\begin{assumption}\label{ass3}
	There exists an $x \in \mathrm{relint}X$ such that $g_i(x)<0, \forall^{m}_{i=1}$.
\end{assumption}
Assumptions \eqref{ass1}-\eqref{ass3} ensure that $x$ is strictly feasible and strong duality holds for the optimization problem \eqref{cvx}.
\begin{remark}
	Note that the gradient function $\nabla g_i(x)$ need not be monotone.
\end{remark}
Let $L:\mathbb{R}^n \times \mathbb{R}^m \rightarrow \mathbb{R}$ define the \textit{Lagrangian} function of the optimization problem \eqref{cvx} as given below
\begin{align}
L(x,\lambda)=f(x)+\lambda^Tg(x). \label{lg}
\end{align}
Let $\lambda_i$ be the Lagrange multipliers associated with $g_i(x)$, then $\lambda\in \Lambda \subseteq \mathbb{R}^m_+=\{\lambda \in \mathbb{R}^m,\lambda_i \geq 0, \forall^m_{i = 1}\}$ defines the corresponding vectors of Lagrange multipliers.

The Lagrangian function $L$ defined in \eqref{lg} is differentiable convex-concave in $x$ and $\lambda$ respectively, i.e., $L(.,\lambda)$ is convex for all $\lambda\in \Lambda$ and $L(x,.)$ is concave for all $x \in X$. 

The saddle-point problem, denoted as $\mathrm{LP}(L,\Omega)$ (not to be confused with a linear programming problem usually designated as $\mathrm{LP}$), finds a pair $x^*\in X$ and $\lambda^*\in \Lambda$ such that the following holds:
\begin{align}
L(x^*,\lambda)\leq L(x^*,\lambda^*)\leq L(x,\lambda^*).\label{spp}
\end{align}
If $x^*$ is the unique minimizer of $L$, then it must satisfy the Karush-Kuhn-Tucker (KKT) conditions stated as follows.
\begin{align}
g_i(x^*)&\leq 0,~\forall^{m}_{i=1} \label{k1}\\
\lambda^*_i&\geq 0,~\forall^{m}_{i=1} \label{k2}\\
\lambda^*_i g_i(x^*)&=0,~\forall^{m}_{i=1} \label{k3}\\
\nabla f(x^*)+\lambda^{*T}\nabla g(x^*)&=0. \label{k5}
\end{align}

Let us define $z = (x,\lambda) \in \Omega = X \times \Lambda$ then $z^*=(x^*,\lambda^*)$ is the saddle point of the Lagrangian function defined in \eqref{lg}.

\begin{remark}
	Since strong duality holds, the KKT conditions \eqref{k1}-\eqref{k5} are necessary and sufficient to guarantee optimality of the problem \eqref{cvx}, with $x^*$ as the unique minimizer of \eqref{cvx} and $z^*$ as the unique saddle-point of \eqref{lg}.
\end{remark} 
Let $G:\mathbb{R}^n \times \mathbb{R}^m \rightarrow\mathbb{R}^{n+m}$ define the gradient map of \eqref{lg} as given below:
\begin{align}
G(z)&=\nabla_z L\\
&=\begin{bmatrix}
\nabla_x L(x,\lambda,\gamma)\\
-\nabla_\lambda L(x,\lambda,\gamma)
\end{bmatrix} \label{gu}
\end{align}
While the primal-dual dynamics for an unconstrained optimization problem corresponds to simply $\dot{z}=-\nabla_z L = -G(z)$, the PD dynamics corresponding to the constrained optimization problem \eqref{cvx}-\eqref{constr_set} is given by:
\begin{align}
\dot{z}=\begin{bmatrix}
-\nabla_x L(x,\lambda,\gamma)\\
[\nabla_\lambda L(x,\lambda,\gamma)]^+_\lambda
\end{bmatrix} \label{pdd}
\end{align}

It is well known that the solution of the PD dynamics \eqref{pdd} coincides with the saddle point $z^*$ of the saddle-point problem \eqref{spp}\cite{arrow1958studies,feijer2010stability}.


This section describes the variational inequality based formulation of the saddle-point problem and proceeds to develop the projected primal-dual dynamics to solve the underlying variational inequality problem.

\subsection{Saddle-point problem as a variational inequality}
The equivalence between variational inequality and constrained optimization problem has been well established in Proposition \ref{optpro} and \cite[Theorem 1.3.1]{facchinei2007finite}]. Using these properties, if $F = \nabla f$ holds, then the equivalence stated in \eqref{r1} proves that the optimal solution $x^*$ of the problem $\mathrm{OPT(f,X)}$ given in \eqref{opt}-\eqref{constr_set} can be obtained by formulating $\mathrm{OPT(f,X)}$ as a variational inequality $\mathrm{VI(\nabla f,X)}$ stated in \eqref{vip}. However, this solution concept holds only with respect to the primal optimizer $x^*$ which is necessary but not sufficient to develop the main results. The similar solution concept must be developed also for the saddle point $z^*$. To this end, this subsection establishes additional properties linking the saddle-point problem $\mathrm{LP}(L,\Omega)$ \eqref{spp} with the equivalent variational inequality problem.


To arrive at the variational inequality problem that is equivalent to the $\mathrm{LP}(L,\Omega)$, the following result from \cite[Proposition 1.3.4]{facchinei2007finite} is useful.
\begin{align}
\mathbf{L}(x,\lambda)=\nabla_xL(x,\lambda) \label{dlL}
\end{align}
where $\mathbf{L}(x,\lambda)$ is defined as the \textit{vector valued} Lagrangian function of the variational inequality $\mathrm{VI(\nabla f,X)}$ stated below:
\begin{align}
\mathbf{L}(x,\lambda)=F(x)+\lambda^T\nabla g(x),~\forall (x,\lambda)\in \Omega.
\end{align}

From \eqref{dlL}, it is obvious that for a given saddle point $z^*$, 
\begin{align}
\mathbf{L}(x^*,\lambda^*)=\nabla_{x^*}L(x^*,\lambda^*) \label{zstar}
\end{align}

Using \eqref{zstar}, an equivalence similar to \eqref{r1} can be obtained between the saddle-point problem $\mathrm{LP}(L,\Omega)$ and the corresponding variational inequality problem of the form $\mathrm{VI(\nabla L,\Omega)}$\footnote[1]{For the sake of notational simplicity, $\nabla_z L$ is denoted as $\nabla L$ in $\mathrm{VI}(\nabla L, \Omega)$ and similar notations}. In line with this, one can define the variational inequality problem $\mathrm{VI(\nabla L,\Omega)}$ by replacing $F$ by $\nabla L$ and $X$ by $\Omega$ in \eqref{vip}, as stated below.
\begin{definition}\label{lgvi}
	For a closed convex set $\Omega \in \mathbb{R}^{n+m}$ and vector function $\nabla L:\Omega\rightarrow \mathbb{R}^{n+m}$, the finite dimensional variational inequality problem, $\mathrm{VI}(\nabla L,\Omega)$, is to determine a vector $z^*\in \Omega$ such that 
	\begin{equation}\small
	(z-z^*)^T\nabla_z L(z^*)\geq 0,~\forall z\in \Omega. \label{mvi}
	\end{equation}
	or equivalently
	\begin{align}\small
	(x-x^*)^T\nabla_x L(x^*,\lambda^*)\geq 0,~\forall x\in X.\\
	(\lambda-\lambda^*)^T(-\nabla_\lambda L(x^*,\lambda))\geq 0,~\forall \lambda\in \Lambda.
	\end{align}
\end{definition}
\begin{remark}
	Denote the boundary and interior of $\Omega$, respectively, by $\partial \Omega$ and $\Omega^0$. If $z \in \partial \Omega$, the set of inward normals to $\Omega$ at $z$ is defined as the dual cone of $T_z\Omega$, as follows:
	
	\begin{align}
	N_z\Omega = \{\eta:\norm{\eta}=1|\langle \eta^T,z-z' \rangle \leq 0, \forall z' \in T_z\Omega\}.
	\end{align} where $T_z\Omega$ is the set of all inward tangent vectors.
	
	From Definition \ref{lgvi}, the necessary and sufficient condition for $z^*$ to be a solution to $\mathrm{VI(\nabla L,\Omega)}$, is that
	\begin{align*}
	-G(z^*) \in N_{z^*}\Omega.
	\end{align*}
\end{remark}
A geometric representation of the variational inequality $\mathrm{VI}(\nabla L,\Omega)$ is given in Fig. \ref{vipro}.
Using Proposition \ref{snle}, the $\mathrm{VI}(\nabla L,\Omega)$ can be proved to be equivalent to solving a system of nonlinear equations, as stated below.
\begin{proposition}\label{snlez}
	Let $\nabla L: \Omega \rightarrow \mathbb{R}^{n+m}$ be a vector function. Then $z^* \in \mathbb{R}^{n+m}$ solves the variational inequality $\mathrm{VI}(\nabla L,\Omega)$ if and only if $z^*$ solves the system of equations
	\begin{align}
	\nabla L(z^*)=\mathbf{0}.\end{align}
	or equivalently (from \eqref{gu})
	\begin{align}
	G(z^*) = \mathbf{0}\label{nlez}.
	\end{align}    
\end{proposition}
\begin{figure}[t]
	\centering
	\includegraphics[width=3.0in]{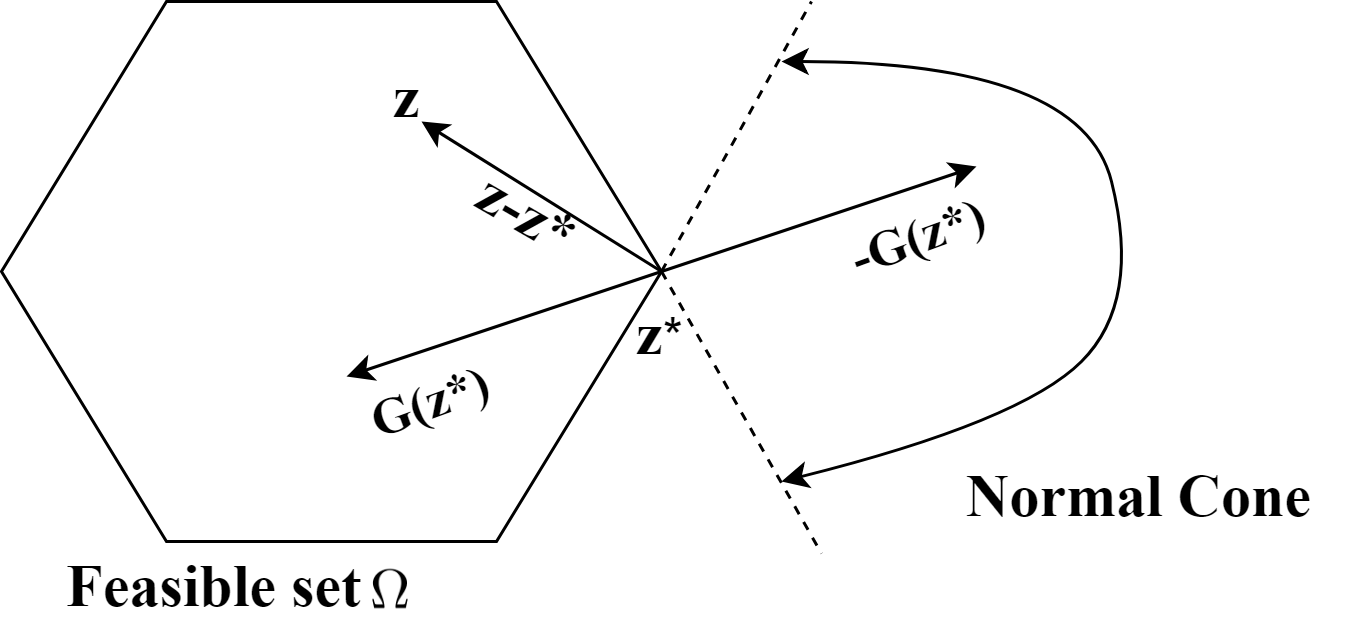}
	\caption{A geometric representation of the variational inequality $\mathrm{VI}(\nabla L,\Omega)$.}
	\label{vipro}
\end{figure}
In a similar way, a fixed point problem corresponding to the $\mathrm{VI}(\nabla L,\Omega)$ can be obtained as given below.
\begin{proposition}\label{lgvifpp}
	$z^*$ is a solution to $\mathrm{VI(\nabla L,\Omega)}$ if and only if for any $\alpha >0$, $z^*$ is a fixed point of the projection map:
	\begin{align}
	z^*=P_\Omega(z^*-\alpha \nabla_z L(z^*)) \label{fppprobz}
	\end{align}
	where \begin{align}
	P_{\Omega}=\arg \min_{v\in \Omega}\norm{z-v}.\label{projectionz}
	\end{align}
\end{proposition}
%
Denote the fixed point problem in Proposition \ref{lgvifpp} as $\mathrm{FPP}(\nabla L,\Omega)$, it follows from the equivalence stated in \eqref{r2} that
\begin{align}
\mathrm{FPP}(\nabla L,\Omega) \implies \mathrm{VI}(\nabla L,\Omega).\label{same route}
\end{align}
Further one can use \cite[Corollary 1.1]{konnov2002theory} to show that the problems \eqref{spp}, \eqref{mvi}, and \eqref{nlez} are equivalent such that the following holds:
\begin{align}\small
\mathrm{FPP}(\nabla L,\Omega) \implies \mathrm{VI}(\nabla L,\Omega) \implies \mathrm{LP}(L,\Omega) .\label{same routelp}
\end{align}
The equivalence result \eqref{same routelp} confirms that the solution $z^*$ of the $\mathrm{VI}(\nabla L,\Omega)$ is the saddle point $z^*$ of saddle-point problem $\mathrm{LP}(L,\Omega)$.
\subsection{Projected primal-dual dynamics}
The framework of finite dimensional variational inequalities studies only the equilibrium solutions, which in a way uses \eqref{fppprobz} to arrive at a static representation of the system \eqref{gu} at its steady state. The dynamic representation of the system \eqref{gu} which shall follow the equivalence stated in \eqref{same route}, must be developed to understand how the variable of interest, i.e., $z$ converges to the solution $z^*$ of the $\mathrm{VI}(\nabla L,\Omega)$. A dynamical model that represents the $\mathrm{VI}(\nabla L,\Omega)$ is widely known as \quotes{the projected dynamical system} (PDS). 
\begin{align}
\dot{z}=\beta\{P_{\Omega}[z-\alpha G(z)]-z\},\label{mpdd}
\end{align}
where $\beta>0$. 
By invoking Remark \ref{remark3} from preliminaries, the $\mathrm{VI}(\nabla L,\Omega)$ can be expressed as a $\mathrm{PDS}(\nabla L,\Omega)$ and the following equivalence can be derived.
\begin{align}\small
\mathrm{PDS}(\nabla L,\Omega) \ScaledImplies \mathrm{FPP}(\nabla L,\Omega) \ScaledImplies \mathrm{VI}(\nabla L,\Omega) \ScaledImplies
\mathrm{LP}(L,\Omega). 
\label{same routelpds}
\end{align}

Towards this end, the $\mathrm{PDS}(\nabla L,\Omega)$ is regarded as the projected PD dynamics whose solution converges to the saddle point solution of the saddle-point problem $\mathrm{LP}(L,\Omega)$. Since $\Omega$ is a closed and convex set, using Theorem \ref{thm2.1nagu} and Theorem \ref{uniqueo} from preliminaries, the existence and the uniqueness of $z^*$ are guaranteed. With this, the next section proceeds towards the stability analysis of the projected PD dynamics.

\subsection{Stability analysis}
In what follows, monotonicity property of $G(z)$ is explored and it is proved that the projected PD dynamics \eqref{mpdd} is Lyapunov stable. 
\begin{lemma}\label{thm1}
	If Assumptions \eqref{ass1}-\eqref{ass3} hold, then $G(z)$ is monotone such that 
	\begin{align}
	[G(z_1)-G(z_2)]^T(z_1-z_2)\geq 0 \label{ineq1}
	\end{align}
	for every pair of $z_1,z_2\in \Omega$.
\end{lemma}
\begin{proof}
	The Jacobian matrix of $G$ is derived as follows:
	\begin{align}\small     
	\nabla G = \begin{bmatrix}
	\nabla^2 f(x) + \lambda^T\nabla^2 g(x) & \nabla g(x)^T\\
	-\nabla g(x)    & \mathbf{0}
	\end{bmatrix}
	\end{align}
	Recall from Proposition \ref{monoto}, $G$ is monotone if and only if the $\nabla G$ is positive semidefinite. By $\nabla G$ positive semidefinite, it is meant that
	\begin{align}
	\frac{1}{2}\nabla G+\frac{1}{2} \nabla G^T \geq 0,\forall z\in \Omega,\forall t.\label{psd}
	\end{align}
	
	Inequality \eqref{psd} can be easily verified by checking the symmetric part of $\nabla G$:
	\begin{align}
	\frac{\nabla G + \nabla G^T}{2} &= \begin{bmatrix}
	\nabla^2 f(x) + \lambda^T\nabla^2 g(x) & \mathbf{0}_{n \times m}\\
	\mathbf{0}_{m \times n}    & \mathbf{0}_{m \times m} 
	\end{bmatrix}\\
	&\geq \mathbf{0}_{(n+m) \times (n+m)},\forall z\in \Omega,\forall t\label{psd1}
	\end{align}
	where $\mathbf{0}$ is a zero matrix of appropriate dimensions. From \eqref{psd1} it is verified that the $\nabla G$ is a $(n+m) \times (n+m)$ matrix which is rank deficient by $m$ rows. Thus it is only a positive semi-definite matrix. Thus from Proposition \ref{monoto}, it can be concluded that $G$ is a monotone.
\end{proof}

\begin{lemma}\label{thm3.5.1}
	Let ${G}(z)$ be continuously differentiable on open convex subset of $\mathbb{R}^{n+m}$. If Assumptions \ref{ass1}-\ref{ass3} hold and $G$ is monotone for all $z\in \Omega$, then with $\alpha>0$, the projected PD dynamics \eqref{mpdd} is Lyapunov stable.
\end{lemma}
\begin{proof}
	For the Lagrangian function \eqref{lg}, the following inequalities always hold:
	\begin{align*}
	{L}(x^*,\lambda^*)-{L}(x^*,\lambda)\geq 0\\
	{L}(x,\lambda^*)-{L}(x^*,\lambda^*)\geq 0
	\end{align*}
	Let us define the Lyapunov function as follows:
	\begin{align}
	V(z) &= ({L}(x^*,\lambda^*)-{L}(x^*,\lambda))+({L}(x,\lambda^*)-{L}(x^*,\lambda^*))\nonumber\\
	&~~~+\frac{1}{2}\norm{z-z^*}^2.\label{Vz}
	\end{align}
	The last term in \eqref{Vz} ensures that $V(z) \geq \frac{1}{2}\norm{z-z^*}^2,\forall z \in \Omega$, thus also ensures the boundedness of the level sets of $V(z)$.
	
	Differentiating $V(z)$ with respect to time $t$ yields:
	\begin{align}
	\dot{V}(z)&=\nabla V(z)\dot{z}\nonumber\\
	&=-[(\nabla {L}(x,\lambda^*)-\nabla {L}(x^*,\lambda))+z-z^*]^T(z-\tilde{z})\nonumber\\
	&=-[G(z)+z-z^*]^T(z-\tilde{z}) \label{ff}
	\end{align}
	Substituting $x = z-\alpha {G}(z)$ and $y=z^*$ in \eqref{bapro}\cite{gao2003exponential}, yields
	\begin{align}
	[z-z^*+\alpha G(z)]^T(z-\tilde{z})\geq \norm{z-\tilde{z}}^2+\alpha (z-z^*)^T{G}(z). \label{use1}
	\end{align}
	Using \eqref{use1} in \eqref{ff} yields,
	\begin{align}
	\dot{V}(z)&\leq -(z-z^*)^T{G}(z)-\norm{z-\tilde{z}}^2\nonumber\\
	&\leq -(z-z^*)^T(G(z)-G(z^*))-\norm{z-\tilde{z}}^2
	\end{align}
	From \eqref{ineq1}, 
	\begin{align}
	\dot{V}(z)\leq 0 \label{leq1}
	\end{align}
	\eqref{ineq1} and \eqref{leq1} ensure that the Lyapunov function \eqref{Vz} is non-increasing along \eqref{mpddr}, hence the projected dynamical system is Lyapunov stable. 
\end{proof}
\begin{remark}
	The Lyapunov stability result from Lemma \ref{thm3.5.1} also holds for the optimization problem of the form \eqref{opt} with linear inequality constraints as defined in the set $X =\{x\in \mathbb{R}^n|Ax\leq b\}$.
\end{remark}
From Proposition \ref{monoto} and Lemma \ref{lemma1.3}, it is observed that the projected PD dynamics \eqref{mpdd} does not achieve exponential (asymptotic) stability since the gradient $G$ is not strongly (strictly) monotone on the Euclidean space. Thus to obtain desired stability results, a geometry other than the Euclidean geometry must be considered. The key to achieving a strongly monotone gradient and a globally exponentially stable PD dynamics is to formulate the problem on a Riemannian manifold as discussed in the subsection below.

\section{Projected primal-dual dynamics over a Riemannian manifold} Proposition \ref{prop1.5} and \ref{prop1.6} establish the relation between monotonicity and convexity on Riemannian manifolds. While the gradient of the Lagrangian function is just monotone on the Euclidean space, using rich metric properties of the Riemannian manifolds, it can be modified to become strongly monotone. On a Riemannian manifold, the key to obtaining a strongly monotone gradient function is to define a Riemannian metric that yields one. Following this reasoning, in this subsection, a linear inequality constrained convex optimization problem is considered over a Riemannian manifold and the Riemannian metric is chosen such that the gradient of the Lagrangian function is strongly monotone. Using the relation between strong monotonicity and uniqueness of the solution, under the assumption of Lipschitz continuity, the projected PD dynamics is shown to have globally exponentially stable saddle-point solution. 

Considers the following optimization problem 
\begin{align}
\min ~&f(x)\nonumber\\
\mathrm{subject~to}~&x \in X \label{lincvx}
\end{align} 
where $X =\{x\in \mathbb{R}^n|Ax\leq b\}$, and $A \in \mathbb{R}^{m \times n}$.
\begin{assumption}\label{as4}
	Let matrix $A$ have full row rank $m\leq n$ and $q_1I\leq AA^T \leq q_2I$, where $I$ is an identity matrix and $q_1,q_2$ are positive constants.
\end{assumption}
The Lagrangian function for the problem \eqref{lincvx} is given by
\begin{align}
L(z)=f(x)+\lambda^T(Ax-b)\label{linlag}
\end{align}where $z= (x,\lambda)$, $z\in \Omega=X \times \Lambda$, $x\in X$, $\lambda\in \Lambda \subseteq \mathbb{R}^m_{\geq 0}$.

The gradient of the Lagrangian \eqref{linlag} is obtained as:
\begin{align}
G(z)&=\nabla_z L\\
&=\begin{bmatrix}
\nabla f(x)+A^T\lambda\\
-(Ax-b)
\end{bmatrix}.\label{gud}
\end{align}

\subsection{Strongly monotone gradient of the Lagrangian}
Consider a smooth manifold $\mathcal{M}\subseteq\mathbb{R}^{m+n}$ and let $\mathcal{M}$ be endowed with Riemannian metric $r$. Define the Lagrangian function $L:\mathcal{M}\rightarrow \mathbb{R}$ then the gradient of $L$ at $z\in \mathcal{M}$ is the unique tangent vector $\mathrm{grad} L$ given as
\begin{align}
\langle \mathrm{grad} L,v\rangle_r=D_zL(v),\forall v\in T_z\mathcal{M}. \label{lgre}
\end{align}
In the matrix notation, \eqref{lgre} implies the following
\begin{align*}
\mathrm{grad}_r L = R^{-1}\nabla L^T \label{gradL}
\end{align*}
where $\nabla L = G(z)$ is the gradient vector of $L$ on Euclidean space $\mathbb{R}^{m+n}$. 

Denote $G_r(z)=\mathrm{grad}_rL$, the differential of $G_r(z)\in \mathcal{S}(\mathcal{M})$ is a linear operator $\mathbb{H}_{G_r}:\mathcal{S}(\mathcal{M})\rightarrow\mathcal{S}(\mathcal{M})$, given by $\mathbb{H}_{G_r}(Y):=\nabla_Y(G_r)$. The linear map $\mathbb{H}_{G_r(z)}:T_z\mathcal{M}\rightarrow\mathcal{M}$ assigned to each point $z\in \mathcal{M}$ is defined by the Hessian of $L$, denoted by $\mathbb{H}_{G_r}(z)v=\nabla_v G_r(z)=R^{-1}\nabla G(z),\forall v\in T_z\mathcal{M}$.

The projection operator $P^r_\mathcal{M}:\mathbb{R}^{n+m}\rightarrow \mathcal{M}$ defined as
\begin{align}
P^r_\mathcal{M}(z)=\arg\min_{v\in T_z\mathcal{M}}\norm{z-v}^2_r.
\end{align}
Correspondingly, the projected PD dynamics on $\mathcal{M}$ is defined as follows:
\begin{align}
\dot{z}=\beta\{P^r_\mathcal{M}[z-\alpha G_r(z)]-z\}.\label{mpddr}
\end{align}
Let the Riemannian metric $R$ be chosen as given below:
\begin{align}
R = \begin{bmatrix}
kI & A^T\\
A & kI
\end{bmatrix}^{-1},\label{P}
\end{align}
where 
\begin{align}
k \geq \sqrt{q_2} \label{kcond0}
\end{align} meets the positive definiteness of the matrix\footnote[2]{A similar matrix representation can also be found in \cite{nguyen2018contraction,qu2019exponential}.} $R$. 
The gradient vector $G_r(z) \in T_z\mathcal{M}$ is given by
\begin{align}
G_r(z) &= R^{-1}G(z),\nonumber\\
&=\begin{bmatrix}
k\nabla f(x)-A^TAx+kA^T\lambda+A^Tb\\
A\nabla f(x)-kAx+AA^T\lambda+kb
\end{bmatrix}\label{newgrad}
\end{align} 

In the following section it is proved that the gradient map \eqref{newgrad} is strongly monotone.
\begin{proposition}\label{thm3.4}
	Consider the problem \eqref{lincvx} and let $(\mathcal{M},r)$ be a $n+m$-dimensional smooth manifold. If Assumption \ref{ass1} and \ref{as4} hold for the problem \eqref{lincvx}, then with the linear map $R^{-1}:T_z\mathcal{M}\rightarrow T_z\mathcal{M}$, the gradient vector $G_r(z)$ is strongly monotone.
\end{proposition}
\begin{proof}
	Recall from Proposition \ref{prop1.5} that for $G_r(z)$ to be strongly monotone, $\nabla_z {G}_r$ must be  positive definite, i.e., for the symmetric part of $\nabla_z {G}_r$, i.e. $\frac{1}{2}\nabla {G}_r+\frac{1}{2}\nabla {G}^T_r$, the following must hold:
	\begin{align}\small 
	\nabla {G}_r+\nabla {G}^T_r&=R^{-1}\nabla G+\nabla G^TR^{-1},\nonumber\\
	&\geq \nu \mathrm{I},\forall z \in \mathcal{M},\forall t\label{gradg}
	\end{align} where $\nu>0$ is a constant, $\mathrm{I}$ is an identity matrix of appropriate dimensions.
	
	The Jacobian of $G_r(z)$, denoted by $\nabla G_r(z)$ is given below:
	\begin{align}\small     
	\nabla G_r(z) = \begin{bmatrix}
	k\nabla^2 f(x) - A^TA & kA^T\\
	A\nabla^2f(x)-kA    & AA^T
	\end{bmatrix}
	\end{align}
	The symmetric part of $\nabla G_r(z)$ is obtained as:
	\begin{align}\small     
	\frac{\nabla {G}_r(z)+\nabla {G}^T_r(z)}{2} = \begin{bmatrix}
	k\nabla^2 f(x) - A^TA & \frac{1}{2}(A\nabla^2f(x))^T\\
	\frac{1}{2}A\nabla^2f(x)    & AA^T
	\end{bmatrix}\label{nablaGr}
	\end{align}
	Let $\mathrm{M}=\nabla {G}_r(z)+\nabla {G}^T_r(z)-q_1\mathrm{I}>0$. Then
	\begin{align}
	\mathrm{M} &= \begin{bmatrix}
	2k\nabla^2 f(x) - 2A^TA -q_1\mathrm{I} & (A\nabla^2f(x))^T\\
	A\nabla^2f(x)    & 2AA^T-q_1\mathrm{I}
	\end{bmatrix}\nonumber\\
	&\geq \begin{bmatrix}
	2k\nabla^2 f(x) - 2A^TA -q_1\mathrm{I} & (A\nabla^2f(x))^T\\
	A\nabla^2f(x)    & AA^T
	\end{bmatrix}\label{Ma}.
	\end{align} 
	Further let $\mathrm{S}=AA^T$, then the Schur compliment of the block $\mathrm{S}$ of the matrix $\mathrm{M}$, denoted by $\mathrm{S}_{\mathrm{Schur}}$ is derived as
	\begin{align}\small 
	\mathrm{S}_{\mathrm{Schur}} &= 2k\nabla^2 f(x) - 2A^TA-q_1I\nonumber\\
	&-(A\nabla^2f(x))^T(AA^T)^{-1}A\nabla^2f(x).\label{m/s}
	\end{align}
	Let $\mathrm{H}=\nabla^2f(x)$ for the notational simplicity. Note that in \eqref{pqrs}, $2k\mathrm{H}>0,\forall k>0$, $2A^TA \geq 0$, $q_1I>0$, and $0\leq \mathrm{H}A^T(AA^T)^{-1}A\mathrm{H} \leq \mathrm{H}^2$. The last terms is a consequence of $A^T(AA^T)^{-1}A\leq I$. Rearranging \eqref{m/s} as given below 
	\begin{align}\small
	2k\mathrm{H} &> 2A^TA+q_1I+\mathrm{H}A^T(AA^T)^{-1}A\mathrm{H}\nonumber\\
	2k\mathrm{H} &> 2A^TA+q_1I+\mathrm{H}^2\label{pqrs}
	\end{align}
	allows to choose $k$ such that $\mathrm{S}_{\mathrm{Schur}}>0$.
	Post multiplying \eqref{pqrs} by $(2H)^{-1}$ yields the following:
	\begin{align}\small
	2kI &>  2A^TA(2H)^{-1}+q_1(2H)^{-1}+\mathrm{H}^2(2H)^{-1}\nonumber\\
	kI &> A^TAH^{-1}+0.5q_1H^{-1}+0.5\mathrm{H}\label{kcond}.\end{align}
	
	Applying Courant-Fischer theorem \cite{horn1990matrix} to \eqref{kcond} yields the following:
	\begin{align}\small
	\lambda_{max}(kI)>\lambda_{max}(A^TAH^{-1}+0.5q_1H^{-1}+0.5\mathrm{H}).\label{mkcond}
	\end{align}
	Since $\lambda_{max}(kI) = k$, \eqref{mkcond} has the following form:
	\begin{align}\small
	k>\lambda_{max}(A^TAH^{-1}+0.5q_1H^{-1}+0.5\mathrm{H}).\label{mkcond1}
	\end{align}
	By choosing $k$ as given in \eqref{mkcond1} ensures that $\mathrm{S}_{\mathrm{Schur}}>0$. 
	But $k$ must also satisfy \eqref{kcond0}, thus $k$ must be chosen such that the following holds:
	\begin{align}
	k > \max\{\sqrt{q_2},\lambda_{max}(A^TAH^{-1}+0.5q_1H^{-1}+0.5\mathrm{H})\} \label{choosek}
	\end{align}
	ensures that both \eqref{kcond0} and \eqref{mkcond} are met. If $k$ is chosen according to \eqref{choosek}, then $\mathrm{S}_{\mathrm{Schur}}>0$ holds such that there exists a $\nu \geq \frac{q_1}{2}$ which implies that 
	\begin{align}
	\langle \mathbb{H}_{G_r(z)}v,v\rangle_r \geq \nu\norm{v}^2_r,\forall v \in T_z\mathcal{M}.
	\end{align}        
	Thus by using Proposition \ref{prop1.5}, the following is derived
	\begin{align}\small
	\langle {G}_r(z_1)-{G}_r(z_2),z_1-z_2\rangle_r\geq \nu \norm{z_1-z_2}^2_r. \label{ffinal}
	\end{align}
	Hence it is proved that $G_r(z)$ is strongly monotone\footnote[3]{Notice that by substituting $\mathrm{M}=\nabla {G}_r(z)+\nabla {G}^T_r(z)$, the strictly monotone map $G_r(z)$ (by definition) can be derived. However, a strongly monotone gradient mapping is also strictly monotone.}.
\end{proof}

\subsection{Exponential stability}
Without loss of generality, let us define $G_r(z)$ similar to \eqref{gu} as follows:
\begin{align}
G_r(z)=\begin{bmatrix}
\nabla^r_x {L}(x,\lambda)\\
-\nabla^r_\lambda {L}(x,\lambda)
\end{bmatrix}, \label{gut}
\end{align} where ${L}(x,\lambda)$ would represent the modified Lagrangian function whose gradient vector field is given by $G_r(z)$. Since, $G_r(z)$ is strongly monotone on $\mathcal{M}$, \eqref{mpddr} will converge to a unique saddle-point solution $z^*$.
\begin{theorem}\label{thm3.5}
	Let $G_r(z)$ be Lipschitz continuous on $D$, then inequality \eqref{ffinal} and $\alpha>0$, imply that the system \eqref{mpdd} with $z(0) \in \mathcal{M}$ is globally exponentially stable at the unique solution $z^*$ of \eqref{mvi}.
\end{theorem}
\begin{proof}
	For each $z(0) \in \mathcal{M}$, there exists a unique solution $z(t)$ of \eqref{mpdd}, that started from $z(0)$. If $[0,t_f)$ is the maximal interval of $z(t)$, then from Lemma \ref{lemma1.3}, $z(t) \in \mathcal{M}$ for all $t \in [0,t_f)$. Since $G_r(z)$ is strongly monotone, the following holds:
	\begin{align*}
	L(x^*,\lambda^*)-L(x^*,\lambda)> 0\\
	L(x,\lambda^*)-L(x^*,\lambda^*)> 0
	\end{align*}
	Let us define the Lyapunov function for the dynamic \eqref{mpddr} as follows:
	\begin{align}
	V_1(z) &= (L(x^*,\lambda^*)-L(x^*,\lambda))+(L(x,\lambda^*)-L(x^*,\lambda^*))\nonumber\\
	&~~~+\frac{1}{2}\norm{z-z^*}^2_r.
	\end{align}
	It is to be noted that $V_1(z)$ possesses a similar structure as that of $V(z)$ defined in \eqref{Vz}, it is also differentiable convex on $\mathcal{M}$, with $V_1(z) \geq \frac{1}{2}\norm{z-z^*}^2_r,\forall z \in \mathcal{M}$, thus bounding all level sets of $V_1(z)$.
	
	Differentiating $V_1(z)$ with respect to time $t$ yields:
	\begin{align}
	\dot{V}_1(z)&=\nabla V_1(z)\dot{z}\nonumber\\
	&=-\langle\nabla L(x,\lambda^*)-\nabla L(x^*,\lambda)+z-z^*,z-\tilde{z}\rangle_r\nonumber\\
	&=-\langle G_r(z)+z-z^*,z-\tilde{z}\rangle_r \label{eff}
	\end{align}
	Substituting $x = z-\alpha G_r(z)$ and $y=z^*$ in \eqref{bapro}\cite{gao2003exponential}, yields
	\begin{align}
	\langle z-z^*+\alpha G_r(z),z-\tilde{z}\rangle_r\geq \norm{z-\tilde{z}}^2_r+\langle\alpha (z-z^*),G_r(z)\rangle_r. \label{euse1}
	\end{align}
	Using \eqref{euse1} in \eqref{eff} yields,
	\begin{align}
	\dot{V}_1(z)&\leq -\langle\alpha (z-z^*),G_r(z)\rangle_r.\label{V1zdot}
	\end{align}
	If $k$ is chosen such that the condition \eqref{mkcond} is satisfied then $G_r(z)$ is strongly monotone.    Using Proposition \eqref{prop1.6} from the preliminary section, the strong monotonicity of $G_r(z)$ also leads to the strong convexity of the Lagrangian function $L(z),\forall z\in \mathcal{M}$, which implies that there exists a unique saddle point $z^*\in \mathcal{M}$, i.e. $\mathcal{M}^* = z^*$. Hence the following inequality can be obtained:
	\begin{align}
	\langle z-&z^*,G_r(z)\rangle_r\geq L(x,\lambda^*)-L(x^*,\lambda)+\frac{\nu}{2}\norm{z-z^*}^2_r, z \in \mathcal{M}.\label{strong conv}
	\end{align}
	
	Using \eqref{strong conv}, \eqref{V1zdot} modifies to the following
	\begin{align}
	\dot{V}_1(z)&\leq -\langle\alpha (z-z^*),G_r(z)\rangle_r\nonumber,\\
	&\leq-\alpha \beta[L(x,\lambda^*)-L(x^*,\lambda)+\frac{\nu}{2}\norm{z-z^*}^2_r],\nonumber\\
	&\leq-\alpha \beta[(L(x^*,\lambda^*)-L(x^*,\lambda))\nonumber\\
	&~~~+(L(x,\lambda^*)-L(x^*,\lambda^*))+\frac{\nu}{2}\norm{z-z^*}^2_r].
	\end{align}
	With $\alpha,\beta>0$, it can be shown that,
	\begin{align}
	\dot{V}_1(z)\leq -\beta\min\{1,\alpha\nu\}V(z).
	\end{align}
	Thus, it is proved that the system \eqref{mpdd} is exponentially stable at the unique solution $z^*$ of \eqref{mvi}.
	Therefor,
	\begin{align}
	\norm{z-z^*}_r\leq ce^{-\beta\frac{\min\{1,\alpha\nu\}}{2}t}
	\end{align}
	where $c=\sqrt{2V_1(z(0))}$.
	
	Further, if $G_r(z)$ is Lipschitz continuous on $\mathcal{M}$, i.e., $\norm{{G}_r(z_1)-{G}_r(z_2)}_r\leq \ell\norm{z_1-z_2}_r,\forall z_1,z_2\in \mathcal{M}$. By using \cite[Theorem 4]{gao2003exponential}, the global exponential stability of the projected PD dynamics can be derived:
	\begin{align}
	\norm{z(t)-z^*}_r\leq \norm{z(0)-z^*}_re^{\frac{-\alpha \beta(4\nu-\alpha \ell^2)}{8}t},\forall t \geq 0. \label{conva}
	\end{align}
	If $\alpha<\frac{4\nu}{\ell^2}$, it follows that the projected PD dynamic \eqref{mpddr} is globally exponentially stable.
\end{proof}
\section{Simulation Results}
This section presents simulation studies of the projected PD dynamics \eqref{mpddr}. It is known that the Euler discretization of the exponentially stable dynamical system owns geometric rate of convergence \cite{stuart1994numerical} for sufficiently small step-sizes. The projected PD dynamics \eqref{mpddr} is Euler discretized with a step size $s>0$ and the following discrete-time projected PD dynamics\cite{nagurney2012projected} is obtained.
\begin{align}
z(\tau+1)=\beta P^r_\mathcal{M}\{z(\tau)-\alpha G_r(\tau)\}.
\end{align}

First example (Example 1) considers an optimization problem of the form \eqref{lincvx} with $m=5$ and $n=10$. The Hessian matrix is assumed to be $H = 20{I}$ with $A$ and $b$ taken as Gaussian random matrix and vector respectively. The distance to the equilibrium point i.e. $z^*=(x^*,\lambda^*)$ for different values of parameter $k$ is shown separately in Fig. \ref{eps1} and \ref{eps12}, where $\varrho=\max\{\sqrt{q_2},\lambda_{max}(A^TAH^{-1}+0.5q_1H^{-1}+0.5\mathrm{H})\}$. It can be seen from the plots that the rate of convergence to the equilibrium points accelerates as the value of $k$ is increased. It implies that increasing the value of $k$ allows increasing the value of $\nu$, which further increases the coefficient of the negative exponential term in \eqref{conva}. The primal optimizers $x^*$ of the problem are also compared to the optimal solution of the same problem obtained using \quotes{quadprog} solver in MATLAB environment as shown in Fig. \ref{eps2}.

In the second example, an $L_2$ regularized least squares problem is considered with $m=30$ and $n=50$. The objective function is $f(x) = \norm{Cx-d}^2_2+\frac{\theta}{2}\norm{x}^2_2$ with $\theta>0$, constrained to $Ax\leq b$. Matrices $(C,A) \in \mathbb{R}^{m \times n}$, and vectors $(d,b)\in \mathbb{R}^{m \times 1}$ are Gaussian random matrices and vectors, respectively. Parameters $\alpha,\beta$ are chosen as unity and the proposed dynamics \eqref{mpddr} is simulated for $k = 1000\max(\varrho)$. A sketch of the error norm as a function of time is shown in Fig. \ref{eps3}. It can be seen that the error norm $\norm{x_i-x^*_i}^2$ has geometric rate of convergence.
\begin{figure}[t]
	\centering
	\includegraphics[width=5in]{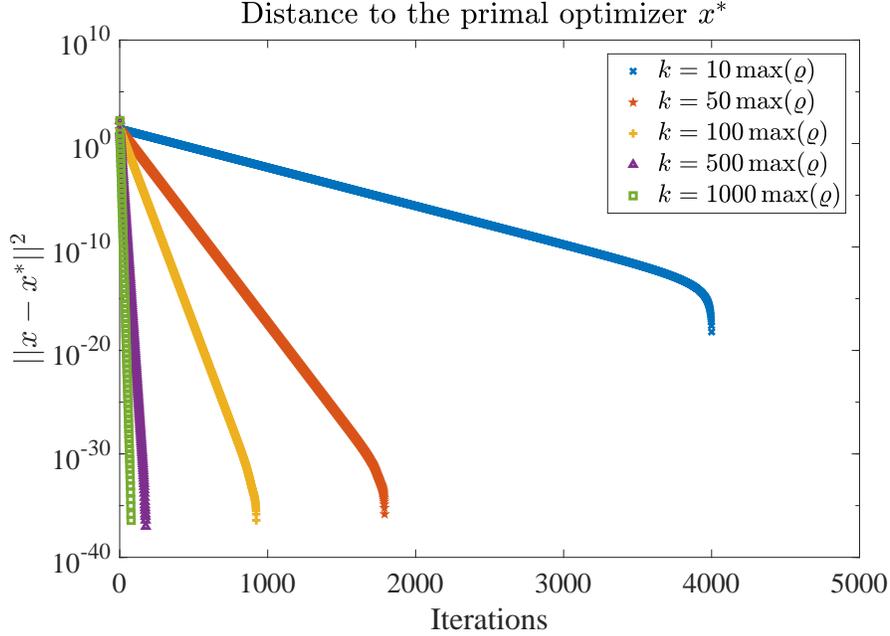}
	\caption{Distance to the primal optimizer $x^*$ for different values of $k$ (Example 1).}
	\label{eps1}
\end{figure}
\begin{figure}[t]
	\centering
	\includegraphics[width=5in]{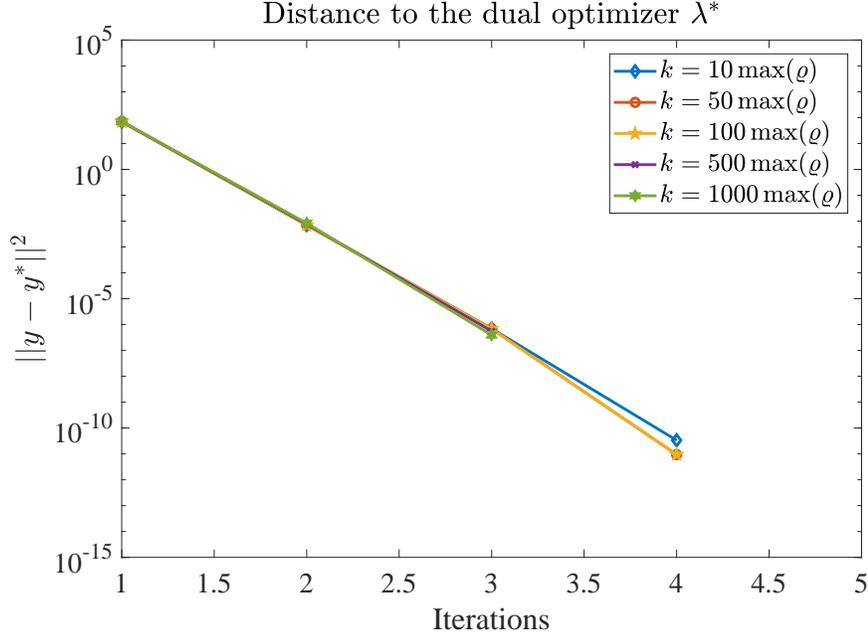}
	\caption{Distance to the dual optimizer $\lambda^*$ for different values of $k$ (Example 1).}
	\label{eps12}
\end{figure}
\begin{figure}[t]
	\centering
	\includegraphics[width=5in]{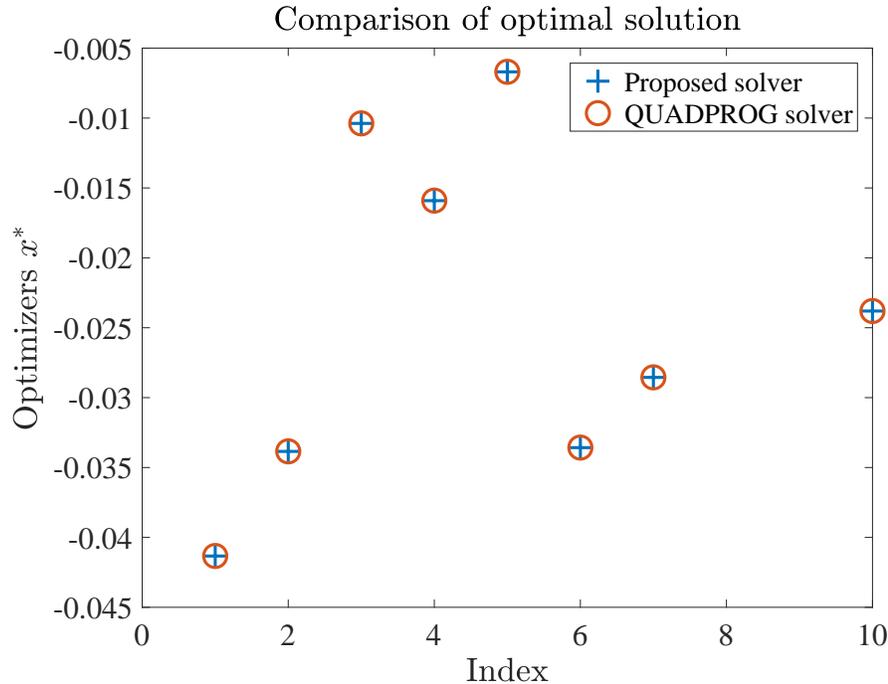}
	\caption{Optimal solution compared with ``{QUADPROG}" (MATLAB) solver.}
	\label{eps2}
\end{figure}
\begin{figure}[t]
	\centering
	\includegraphics[width=5in]{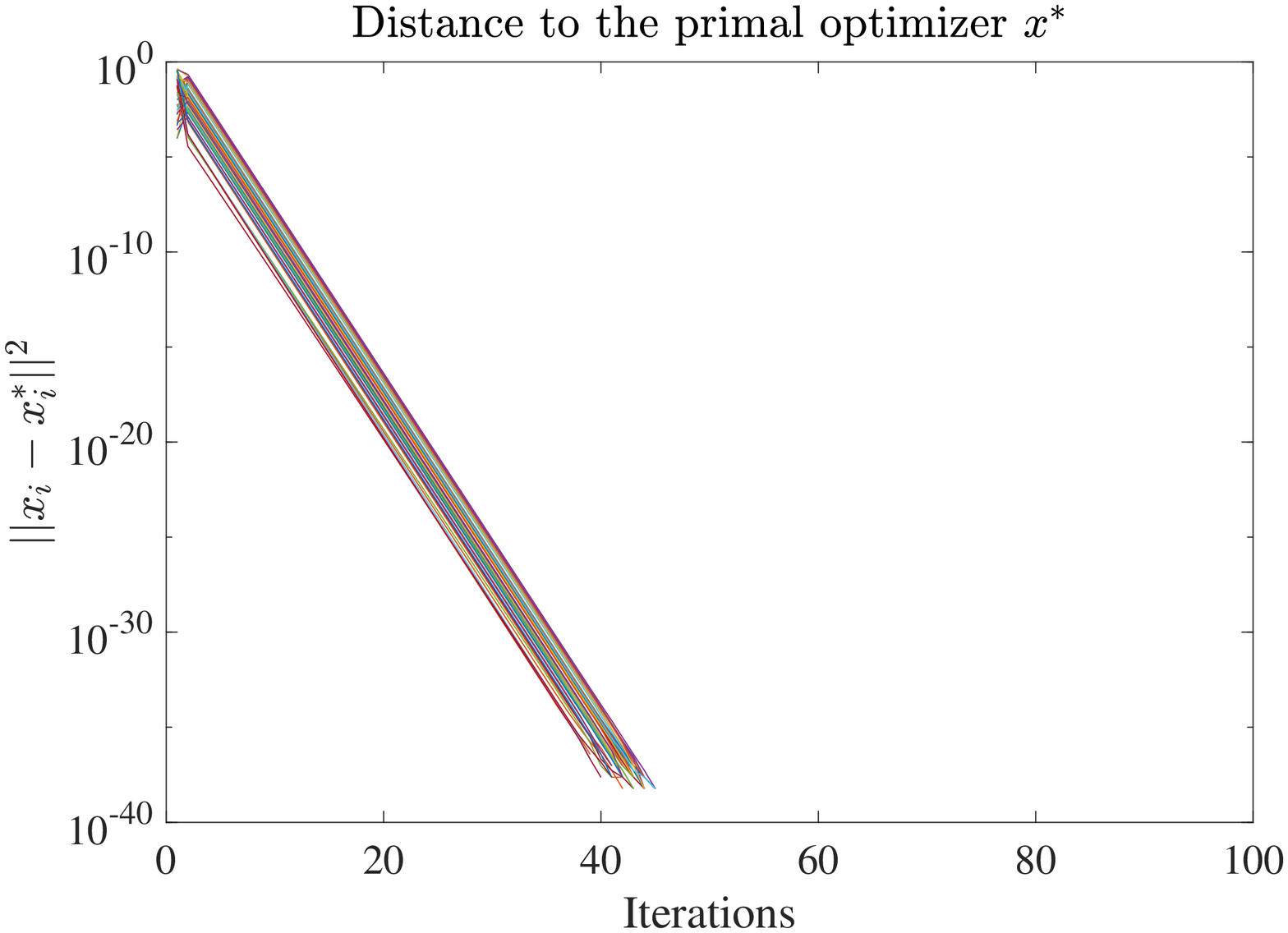}
	\caption{Distance to the primal optimizer $x^*$ ($L_2$ regularized least squares problem).}
	\label{eps3}
\end{figure}
\section{Conclusions and discussion}\label{concl}
This paper has proposed a projected dynamical system based formulation of the saddle point problem to solve a constrained convex optimization problem. Monotonicity properties of the gradient of the underlying Lagrangian function are evaluated. It is found that this gradient map is just monotone on the Euclidean space which restricts the proposed dynamics from being globally exponentially stable. This has lead to a saddle-point solution of the proposed dynamics which is only Lyapunov stable. It confirmed that the desired property concerning strong monotonicity of the underlying gradient map ceases to exist on the Euclidean space. Due to which the exponential stability of the proposed dynamics cannot be obtained. It further compelled to search for a differential geometry where such properties are obtainable. To this end, our results proved that the proposed dynamics is exponentially stable on a Riemannian manifold whose Riemannian metric is chosen such that the underlying gradient is strongly monotone. Then it is shown that the exponential stability holds globally under the Lipschitz continuity of the gradient map. However, the analysis pertains to a linear inequality constrained convex optimization problem. 

There are many network-based optimization problems that fall under the category of linear inequality constrained optimization, one of such problems is distributed support vector machines\cite{forero2010consensus,stolpe2016distributed} which is solved in a distributed manner over a network of nodes acquiring valuable statistics. The results of this paper can be further extended to such problems.

The proposed approach can also be generalized to a convex optimization problem with convex inequality constraints under regularity conditions. However, it will not be a straightforward extension of the present work. It is expected that the underlying mathematical framework would require additional properties concerning convexity of the inequality constraints which is left as future scope of this paper.
\bibliographystyle{unsrt}
\bibliography{dsvm_pd} 
\end{document}